\documentclass{amsart}  

\usepackage{amssymb,amscd}
\usepackage[all,arc]{xy}
\usepackage[english]{babel}

\usepackage{tikz}
\newtheorem{thm}{Theorem}
\newtheorem*{thm*}{Theorem}

\newtheorem*{thmA*}{Theorem A}
\newtheorem*{thmB*}{Theorem B}
\newtheorem*{thmC*}{Theorem C}

\newtheorem{cor}{Corollary}
\newtheorem{prop}{Proposition}
\newtheorem{lem}{Lemma}

\theoremstyle{definition}
\newtheorem*{defn}{Definition}

\newtheorem*{exmp*}{Example}

\newcommand{\conv}{\mathop{\rm conv}}
\newcommand{\inter}{\mathop{\rm int}}
\newcommand{\relint}{\mathop{\rm relint}}
\newcommand{\rk}{\mathop{\rm rk}}
\newcommand{\orb}{\mathop{\rm orb}}
\newcommand{\Cone}{\mathop{\rm Cone}}

\title{The Hartogs extension phenomenon in toric varieties}

\author{Sergey Feklistov, Alexey Shchuplev}

\date{\today}

\begin{document}

\begin{abstract}
We study the Hartogs extension phenomenon in non-compact toric varieties and its relation to the first cohomology group with compact support. We show that a toric variety  admits this phenomenon if at least one connected component of the fan complement is concave, proving by this an earlier conjecture M. Marciniak.
\end{abstract}

\maketitle

\section{Introduction}
\label{s:0}
The classical Hartogs extension theorem states that for every domain $D\subset\mathbb{C}^{n}(n>1)$ and a compact set $K\subset D$ such that $D\setminus K$ is connected, any holomorphic function  $f$ on $D\setminus K$ extends holomorphically to $D$. A natural question arises if this is true for complex analytic spaces. 

\begin{defn}
\emph{We say that a connected complex space $X$ admits the Hartogs phenomenon if for any domain $D\subset X$ and a compact set $K\subset D$ such that $D\setminus K$ is connected, the restriction homomorphism $$H^{0}(D,\mathcal{O})\to H^{0}(D\setminus K, \mathcal{O})$$ is an isomorphism. }
	\end{defn}
In this or a similar formulation this phenomenon has been extensively studied in many situations, including Stein manifolds and spaces, $(n-1)$-complete normal complex spaces and so on \cite{Andersson,AndrGrau,AndrHill,AndrVesen,BanStan,ColtRupp,Harvey,Merker,Vassiliadou,Rossi}.

Our goal is to study the Hartogs phenomenon in toric varieties. Toric varieties is a class of algebraic varieties that are rather easy to construct. Locally, they are algebraic sets defined by systems of binomial equations which are glued together via monomial mappings. This allows to describe their structure and properties in purely combinatorial terms; each $p$-dimensional toric variety $X$ is encoded by a fan $\Sigma$ --- a collection of cones  in $\mathbb R^p$ with the common apex that may intersect only along their common face. The toric geometry arises naturally in constructions involving monomial functions or curves, notably in resolution of singularities or asymptotics, e.g. in the theory of residue currents (see, e.g., \cite{STY,Tsikh, Alekos}).

In the context of toric varieties the Hartogs phenomenon was first studied by M.~A.~Marciniak in her thesis \cite{Marci} and a paper in this journal \cite{Marci3}. She was able to prove the global version of the Hartogs phenomenon (with $D=X$):
	\begin{thm*}
	Let $X_{\Sigma}$ be a smooth toric surface or the total space of a smooth toric line bundle over a compact base. If the support $|\Sigma|$ of $\Sigma$ is a strictly convex cone then for any compact set $K\subset X_\Sigma$ such that $X_{\Sigma}\setminus K$ is connected, the restriction homomorphism $H^{0}(X_{\Sigma},\mathcal{O})\to H^{0}(X_{\Sigma}\setminus K, \mathcal{O})$ is an isomorphism.
	\end{thm*}
To prove the theorem one needs to know the transition functions between coordinate charts to construct holomorphic extensions of a function given in one chart to all others. Since the transition functions are given by the cones of maximal dimensions in the fan, this yields the restrictions on $\Sigma$. However, this approach can hardly be generalised to higher dimensions because of increasing combinatorial difficulties. Nevertheless, Marciniak has formulated the following   

\noindent {\bfseries Conjecture \cite{Marci}}. \textit{Let $X_\Sigma$ be a smooth toric variety. If the complement of $|\Sigma|$ has at least one concave connected component then $X_\Sigma$ admits the global Hartogs phenomenon.}

We shall follow a more general approach that goes back to Serre \cite{Serre}. First we prove the result about vanishing cohomology
 	\begin{thmA*}
	Let $X_{\Sigma'}$ be a $p$-dimensional toric variety with the fan $\Sigma'$. Assume that the complement of the fan's support  $C:=\mathbb{R}^{p}\setminus |\Sigma'|$ is connected, then  $H^{1}_{c}(X_{\Sigma'},\mathcal{O})=0$ if and only if $\conv(\overline{C})=\mathbb{R}^{p}$.
	\end{thmA*}
This allows us to specify what concavity in the conjecture formulated above means. Let $\Sigma$ be a fan in $\mathbb{R}^{p}$, and $\mathbb{R}^{p}\setminus |\Sigma|$ be its complement. The complement may have several connected components $\mathbb{R}^{p}\setminus |\Sigma|=\bigsqcup\limits_{j} C_{j}$.
	\begin{defn}
A complement component $C_{j}$ is called concave if $$\conv(\overline{C_j})=\mathbb{R}^{p}.$$
	\end{defn}	

Then we use the toric version of Serre's theorem
\begin{thmB*}
 	Let $X_{\Sigma}$ be a non-compact normal toric variety with the complement $\mathbb{R}^{p}\setminus |\Sigma|$ being connected. The cohomology group $H^{1}_{c}(X_{\Sigma},\mathcal{O})$ is trivial if and only if  $X_{\Sigma}$ admits the Hartogs phenomenon.
 \end{thmB*} 
\noindent to prove the main result from which Marciniak's conjecture follows
	 
 	\begin{thmC*}
Let $X_{\Sigma}$ be a normal non-compact toric variety with the fan $\Sigma$ whose complement is $\mathbb{R}^{p}\setminus |\Sigma|=\bigsqcup\limits_{j=1}^{n}C_{j}$. Then
	\begin{itemize}
	\item if at least one of $C_j$'s is concave then $X_{\Sigma}$ admits the Hartogs phenomenon.
	\item if $n=1$ then the converse is also true, i.e. if $X_{\Sigma}$ admits the Hartogs phenomenon then $\mathbb{R}^{p}\setminus |\Sigma|$ is concave.
	\end{itemize}  
	\end{thmC*}

\section{Toric varieties}

In this section we briefly review the necessary elements of the theory of toric varieties. Here we follow \cite{Cox} and \cite{Oda}.

\subsection{Lattices and cones}	

	Consider a lattice $N$ of rank $r$ and  the dual lattice  $M:=Hom_{\mathbb{Z}}(N,\mathbb{Z})$ with the canonical $\mathbb{Z}$-bilinear product $$\langle-,-\rangle\colon M\times N\to \mathbb{Z}, \langle l,x\rangle:=l(x).$$ We define scalar extensions  $$N_{\mathbb{R}}:=N\otimes_{\mathbb{Z}}\mathbb{R}, M_{\mathbb{R}}:=M\otimes_{\mathbb{Z}}\mathbb{R}$$ and obtain  the canonical $\mathbb{R}$-bilinear product $$\langle-,-\rangle\colon M_{\mathbb{R}}\times N_{\mathbb{R}}\to \mathbb{Z}, \langle l,x\rangle:=l(x).$$
	
	\begin{defn}
A subset $\sigma\subset N_{\mathbb{R}}$ is a rational polyhedral cone with the apex at the origin $O\in N_{\mathbb{R}}$ if there is a finite number of  $n_{1},n_{2},\cdots, n_{s}\in N$ such that $$\sigma=\{a_{1}n_{1}+\cdots a_{s}n_{s}\mid a_{i}\in\mathbb{R},a_{i}\geq0,\forall\,i=1,\cdots,s\}.$$ 
	A rational polyhedral cone  $\sigma$ is strictly convex if  $\sigma\cap(-\sigma)=\{O\}$.
	\end{defn}
Throughout the paper, unless stated explicitly, we shall refer to rational polyhedral cones simply as cones. A cone generated by the list $\{n_{1},n_{2},\cdots, n_{s}\}$ we denote as $\Cone(n_{1},n_{2},\cdots, n_{s})$. 
	
	\begin{defn}\indent
	\begin{itemize}
	\item The dimension $\dim\sigma$ of the cone $\sigma$ is the dimension of the smallest subspace $\mathbb{R}(\sigma)\subset N_{\mathbb{R}}$ containing  $\sigma$.
	\item The dual cone $\sigma^{\vee}\subset M_{\mathbb{R}}$ for $\sigma$ is defined as $$\sigma^{\vee}:=\{x\in M_{\mathbb{R}}\mid\langle x,y\rangle\geq0,\forall\,y\in\sigma\}.$$
	\item A face $\tau$ of $\sigma$ (denoted as $\tau\prec\sigma$) is the set $$\tau:=\sigma\cap H_{m_{0}}$$ for some $m_{0}\in\sigma^{\vee}$ where $H_{m_{0}}:=\{y\in N_{\mathbb{R}}\mid \langle m_{0},y\rangle=0\}.$
	\end{itemize}	
	\end{defn}

A strictly convex rational polyhedral cone $\sigma\subset N_{\mathbb{R}}$ defines a semigroup $S_{\sigma}:=\sigma^{\vee}\cap M$. Its properties are summarized in the following
	\begin{prop}\cite[Proposition 1.1]{Oda}
	Let $\sigma\subset N_{\mathbb{R}}$ be a strictly convex rational polyhedral cone. Then the following hold
	\begin{enumerate}
	\item $S_{\sigma}\subset M$ is an additive subsemigroup containing the origin, i.e. $O\in S_{\sigma}$ and for all $m,m'\in S_{\sigma}$ we have $m+m'\in S_{\sigma}$.
	\item $S_{\sigma}$ is finitely generated as an additive semigroup, i.e. there exist $m_{1},\cdots, m_{p}\in S_{\sigma}$ such that $S_{\sigma}=\mathbb{Z}_{\geq0}\langle m_{1},\cdots,m_{p}\rangle$.
	\item $S_{\sigma}+(-S_{\sigma})=M$.
	\item $S_{\sigma}$ is saturated, i.e. if $cm\in S_{\sigma}$ for $m\in M$ and $c\in\mathbb{Z}_{>0}$ then $m\in S_{\sigma}$.
	\end{enumerate}
	Conversely, for any additive semigroup $S\subset M$ satisfying 1--4 there exists a unique strictly convex rational polyhedral cone $\sigma\in N_{\mathbb{R}}$ such that $S=S_{\sigma}$.
	\end{prop}

	\subsection{Fans and toric varieties}
		
	Consider an algebraic torus $T_{N}$ associated with the lattice $N$ (and the dual lattice $M$)
	 $$T_{N}:=Hom_{\mathbb{Z}}(M,\mathbb{C}^{*})=N\otimes_{\mathbb{Z}}\mathbb{C}^{*}.$$
Each element $m\in M$ gives rise to a character (a group homomorphism)  $$\chi^{m}\colon T_{N}\to\mathbb{C}^{*}, \chi^{m}(t):=t(m).$$ 
The characters form a multiplicative group $T_{N}$ which may be identified with the lattice $M$. Each element $n\in N$ defines a one-parametric subgroup $$\lambda^{n}\colon\mathbb{C}^{*}\to T_{N}, \lambda^{n}(t)(m):=t^{\langle m,n \rangle}.$$ 
The multiplicative group of all one-parametric subgroups of the torus $T_{N}$ may be identified with the lattice $N$.

 Let $\{n_{1},\cdots,n_{r}\}$ be a $\mathbb{Z}$-basis of the lattice $N$ and $\{m_{1}.\cdots,m_{r}\}$ be the dual $\mathbb{Z}$-basis of $M$. Denoting $u_{j}=\chi^{m_{j}}$, we get an isomorphism
	\begin{gather*}
	T_{N} \cong (\mathbb{C}^{*})^{r},\\ t \mapsto (u_{1}(t),\cdots,u_{r}(t)).
	\end{gather*}
Now, if $m=\sum\limits_{j=1}^{r}a_{j}m_{j}$ then $\chi^{m}(t)=u_{1}^{a_1}\hdots u_{r}^{a_r}$; if $n=\sum\limits_{j=1}^{r}b_{j}n_{j}$ then $\lambda^{n}(t)=(t^{b_1},\cdots,t^{b_r})$.
	
For any cone $\sigma$ we construct a normal complex space $U_{\sigma}$ as follows
	\begin{prop}\cite[Proposition 1.2]{Oda}\label{affvar}
	Let $$S_{\sigma}=\mathbb{Z}_{\geq0}\langle m_{1},\cdots,m_{p}\rangle \subset M$$ be a finitely generated subsemigroup determined by a strictly convex rational polyhedral cone $\sigma\subset N_{\mathbb{R}}$. Let $$U_{\sigma}:=\{u\colon S_{\sigma}\to\mathbb{C}\mid u(O)=1, u(m+m')=u(m)u(m'),\forall\,m,m'\in S_{\sigma}\}$$ and $\chi^{m}(u):=u(m)$ for $m\in S_{\sigma}$,  $u\in U_{\sigma}$. Then 
	\begin{itemize}
	\item the mapping $(\chi^{m_1},\cdots,\chi^{m_p})\colon U_{\sigma}\to\mathbb{C}^{p}$ is injective. 
	\item identifying $U_{\sigma}$ with its image we get an algebraic subset in $\mathbb{C}^{p}$ given by a system of equations of the form $$z_{1}^{a_{1}}\cdots z_{p}^{a_{p}}-z_{1}^{b_{1}}\cdots z_{p}^{b_{p}}=0$$ for $(a_{1},\hdots,a_{p}),(b_{1},\hdots,b_{p})\in\mathbb{Z}_{\geq0}^{p}$ such that $\sum\limits_{i=1}^{p}a_{i}m_{i}=\sum\limits_{i=1}^{p}b_{i}m_{i}$.
	\item the structure of an $r$-dimensional irreducible normal complex space on $U_{\sigma}$ is induced by the complex analytic structure of  $\mathbb{C}^{p}$ and does not depend on the choice of the system of generators $\{m_{1},\cdots,m_{p}\}$ for the subsemigroup $S_{\sigma}$.
	\item each $m\in S_{\sigma}$ defines a polynomial function $\chi^{m}$ on $U_{\sigma}$ which is holomorphic with respect to the induced complex structure.
	\end{itemize}
	\end{prop}
	
	Any face $\tau$ of a cone $\sigma$ is again a cone, hence, it defines a complex space $U_{\tau}$. The relationship between $U_{\sigma}$ and $U_{\tau}$ is given by	
	\begin{prop}\cite[Proposition 1.3]{Oda}\label{faceandopen}\indent
	\begin{enumerate}	
	\item For a strictly convex rational polyhedral cone $\sigma\subset N_{\mathbb{R}}$ its dual cone $\sigma^{\vee}\subset M_{\mathbb{R}}$ is a rational polyhedral cone.
	\item If $\tau\prec\sigma$ then there exists $m_{0}\in M\cap\sigma^{\vee}$ such that
		\begin{itemize}
		\item $\tau=\sigma\cap H_{m_0}$.
		\item $S_{\tau}=S_{\sigma}+\mathbb{Z}_{\geq0}(-m_{0})$.
		\item $U_{\tau}=\{u\in U_{\sigma}\mid u(m_{0})\neq0\}$ is an open subset in $U_{\sigma}$.
		\end{itemize}
	\end{enumerate}
	\end{prop}

If $\sigma_1,\sigma_2$ are two cones with a common face $\tau=\sigma_{1}\cap\sigma_{2}$, then $U_{\tau}$ is an open subset of both $U_{\sigma_1}$ and $U_{\sigma_2}$ and we can glue them together along  $U_{\tau}$ to obtain a complex space. A fan is a collection of cones specifying how to glue all $U_\sigma$'s together.
	
	\begin{defn}
A fan is a pair	$(\Sigma,N)$ where $\Sigma$ is a finite set of strictly convex rational polyhedral cones $\sigma\subset N_{\mathbb{R}}$ with the properties: 
	\begin{itemize}
	\item each face of any cone  $\sigma\in\Sigma$ is in $\Sigma$.
	\item for any cones $\sigma,\sigma'\in\Sigma$ their intersection $\sigma\cap\sigma'$ is a face of both.
	\end{itemize}
	\end{defn}	
	
\begin{exmp*}
		Consider a collection of the following 2-dimensional cones 
\begin{align*}
&\Cone((1,\,1,\,1),\,(1,\,-1,\,-1)), & &\Cone((1,\,-1,\,-1),\,(-1,\,-1,\,1)),\\
&\Cone((-1,\,-1,\,1),\,(-1,\,1,\,-1)), & &\Cone((-1,\,1,\,-1),\,(1,\,1,\,1))
\end{align*}
together with all their faces. This collection is a fan. Both complement components of this fan are concave.	
	\begin{figure}[h!]
	\begin{center}
	\includegraphics[width=0.4\textwidth]{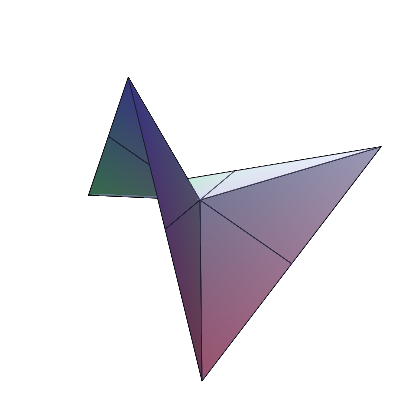}
	\label{3fan1}
	\caption{A fan with concave complement components.}
	\end{center}					
	\end{figure} 
\end{exmp*}

	\begin{prop}\cite[Theorem 1.4]{Oda}
	For a fan $(\Sigma,N)$ the result of gluing $\{U_{\sigma}\mid\sigma\in\Sigma\}$ together as described above is an irreducible normal complex space $X_{\Sigma,N}$ of dimension $r=\rk{N}$.
	\end{prop}
	
	\begin{defn}
The complex space $X_{\Sigma,N}$ is called a toric variety associated to the fan $(\Sigma,N)$.
	\end{defn}

When the lattice $N$ is fixed we omit mentioning it in the notation, writing simply $\Sigma$ and $X_\Sigma$. We also denote by $\Sigma(1)$ the set of all 1-dimensional cones in $\Sigma$.

Smoothness and compactness of a normal toric variety can also be described in terms of its fan.
	\begin{defn}
We call a cone $\sigma$ smooth if there exists a $\mathbb{Z}$-basis $\{n_{1},\cdots,n_{r}\}$ of the lattice $N$ and $s\leq r$ such that $\sigma=\Cone(n_{1},\cdots,n_{s})$.
	\end{defn}
	\begin{thm}\cite[Theorem 1.10]{Oda}
A toric variety $X_{\Sigma}$ is smooth (i.e. is a smooth complex manifold) if and only if every cone  $\sigma\in\Sigma$ is smooth.
	\end{thm}
	\begin{defn}
A fan $\Sigma$  is called complete if its support $|\Sigma|:=\bigcup\limits_{\sigma\in\Sigma}\sigma$ coincides with $N_{\mathbb{R}}$.
	\end{defn}
	\begin{thm}\cite[Theorem 1.11]{Oda}
	A toric variety $X_{\Sigma}$ is compact if and only if its fan $\Sigma$ is complete.
	\end{thm}

\subsection{Torus action}
	Let $(\Sigma, N)$ be a fan; it always contains a zero-dimensional cone $\{O\}$ --- the origin. Then $S_{\{O\}}=M$, and $U_{\{O\}}=T_{N}$. This cone is a face of every other cone in $(\Sigma, N)$, therefore $T_{N}$ is an open subset in each $U_{\sigma}$. Hence, a toric variety $X_{\Sigma}$ contains an algebraic torus $T_{N}$ as an open subset.
	
	The algebraic torus $T_{N}$ acts algebraically on $X_{\Sigma}$ (by this we mean a morphism of algebraic varieties $T_{N}\times X\to X$). The action is defined as follows:
	\begin{enumerate}	
	\item Let $t\in T_{N}$, it gives a group homomorphism $t\colon M\to\mathbb{C}^{*}$;
	\item If $u\in U_{\sigma}$, then we have a mapping  $u\colon S_{\sigma}\to\mathbb{C}$ with the property $u(O)=1, u(m+m')=u(m)u(m'),\forall\,m,m'\in S_{\sigma}\}$;
	\item Define $tu\colon S_{\sigma}\to\mathbb{C}$ as $(tu)(m):=t(m)u(m)$ for $m\in S_{\sigma}$;
	\item $tu\in U_{\sigma}$,  since $tu(O)=1$ and $(tu)(m+m')=t(m+m')u(m+m')=t(m)t(m')u(m)u(m')=(tu)(m)(tu)(m')$;
	\item Since $(t_{1}t_{2})(u)=t_{1}(t_{2}(u))$, we get an action of $T_{N}$ on each $U_{\sigma}$  and $X_{\Sigma}$.
	\end{enumerate}
	
Thus, a toric variety $X_{\Sigma}$ contains an algebraic torus $T_{N}$ with an algebraic action on itself which extends to an action on the whole $X_{\Sigma}$. The converse is also true
	\begin{thm}\cite[Theorem 1.5]{Oda}
	Suppose the algebraic torus $T_{N}$ acts algebraically on an irreducible normal algebraic variety $X$ locally of finite type over $\mathbb{C}$. If $X$ contains an open orbit isomorphic to  $T_{N}$, then there exists a unique fan $(\Sigma,N)$ such that $X$ is equivariantly isomorphic to $X_{\Sigma}$.
	\end{thm}

For each cone $\sigma\in\Sigma$ we can consider an algebraic torus $$O(\sigma):=\{\text{group homomorphisms }u\colon M\cap\sigma^{\perp}\to\mathbb{C}^{*}\}$$ as a $T_{N}$-orbit in $X_{\Sigma}$. In terms of the fan $\Sigma$ we can describe the $T_{N}$-orbits as follows:
	\begin{prop}\cite[Proposition 1.6]{Oda}\indent\label{orbconecoresp}
	\begin{enumerate}
	\item Each $T_{N}$-orbit is of the form $O(\sigma)$. 
	\item The fan $\Sigma$ is in a 1-1 correspondence with the set of $T_{N}$-orbits in $X_{\Sigma}$. 
	\item $O(\{O\})=U_{\{O\}}=T_{N}$.
	\item For $\sigma\in\Sigma$, the complex dimension $\dim(O(\sigma))=r-\dim(\sigma)$.
	\item For $\tau, \sigma\in\Sigma$,  $\tau\prec\sigma$ if and only if $O(\sigma)\subset\overline{O(\tau)}$.
	\item For $\sigma\in\Sigma$, the orbit $O(\sigma)$ is the only closed $T_{N}$-orbit in $U_{\sigma}$, and $U_{\sigma}=\bigsqcup\limits_{\tau\prec\sigma}O(\tau)$.
	\item Let $n\in N$ and $\sigma\in\Sigma$. Then $n\in\sigma$ if and only if for the one-parameter subgroup $\lambda^{n}$ corresponding to $n$ there exists the limit $\lim\limits_{t\to0}\lambda^{n}(t)$ in $U_{\sigma}$. In this case, the limit $\lim\limits_{t\to0}\lambda^{n}(t)$ coincides with the identity element of the algebraic torus $\orb(\tau)$, where $\tau$ is a face of $\sigma$ which contains $n$ in its relative interior. 
	\end{enumerate}
	\end{prop}

\subsection{Equivariant compactification and mappings}

In what follows, we consider toric varieties with incomplete fans (i.e. non-compact toric varieties). It turns out that such varieties can be compactified such that the resulting variety is again a toric variety. This is given by Sumihiro's theorem 
	\begin{thm}\cite[p. 17]{Oda}\label{equvcompact}
Suppose a connected linear algebraic group  $G$ acts algebraically on an irreducible normal algebraic variety $X$ of finite type over $\mathbb{C}$. Then $X$ can be embedded as a  $G$-invariant open subset of a complete irreducible normal algebraic variety $\widehat{X}$ on which $G$ acts algebraically.
	\end{thm}
We have	
	\begin{cor}\label{corequvcompact}
	For any fan $(\Sigma,N)$  there exists a complete fan $(\widehat{\Sigma},N)$ such that $X_{\widehat{\Sigma}}$ is a $T_{N}$-equivariant compactification of $X_{\Sigma}$.
	\end{cor}		

Holomorphic mappings compatible with torus actions can also be described in terms of fans. Let  $(\Sigma,N)$ and $(\Sigma',N')$ be two fans.
	\begin{defn}
	A fan morphism $\phi\colon (\Sigma',N')\to(\Sigma,N)$ is a $\mathbb{Z}$-linear homomorphism of lattices $\phi\colon N'\to N$ such that the corresponding scalar extension $\phi_{\mathbb{R}}\colon N_{\mathbb{R}}'\to N_{\mathbb{R}}$ has the property: for each $\sigma'\in\Sigma'$ there exists $\sigma\in\Sigma$ such that $\phi_{\mathbb{R}}(\sigma')\subset\sigma$.
	\end{defn}
The following theorem describes morphisms of toric varieties
	\begin{thm}\cite[Theorem 1.13]{Oda}
	\begin{itemize}
	\item A fan morphism $\phi\colon (\Sigma',N')\to(\Sigma,N)$ gives rise to a holomorphic mapping $$\phi_{*}\colon X_{\Sigma',N'}\to X_{\Sigma,N}$$ whose restriction to the open subset $T_{N'}$ coincides with the homomorphism of algebraic tori $$\phi\otimes1\colon T_{N'}\to T_{N}.$$
	Besides, $\phi_{*}$ is equivariant with respect to the actions of $T_{N'}$ and $T_{N}$ on the toric varieties.
	\item Conversely, let $f'\colon T_{N'}\to T_{N}$ be a homomorphism of algebraic tori, and $f\colon X_{\Sigma',N'}\to X_{\Sigma,N}$ be a holomorphic mapping equivariant with respect to $f'$. Then there exists a unique $\mathbb{Z}$-linear homomorphism $$\phi\colon N'\to N,$$ which gives rise to a fan morphism $$\phi\colon (\Sigma',N')\to(\Sigma,N)$$ such that $f=\phi_{*}$.
	\end{itemize}
	\end{thm}
In other words, an equivariant holomorphic mapping of normal toric varieties in local coordinates is given by monomial functions.

\subsection{Resolution of singularities and subdivisions of fans}\label{equresolution}

	\begin{defn}\indent
	\begin{itemize}
	\item Let $(\Sigma,N)$ and $(\Sigma',N)$ be two fans. A fan $(\Sigma',N)$ is called a subdivision of a fan  $(\Sigma,N)$ if each cone of  $\Sigma'$ is contained in a cone of $\Sigma$, and $|\Sigma'|=|\Sigma|$.
	\item We call a subdivision $(\Sigma',N)$ smooth if each cone $\sigma'\in\Sigma'$ is smooth.
	\end{itemize}
	\end{defn}
	
Consider the identity mapping of the lattice $id\colon N\to N$. This mapping is a fan morphism  $id\colon (\Sigma',N)\to(\Sigma,N)$, since $\Sigma'$ is a subdivision of  $\Sigma$, and we get an equivariant holomorphic mapping $$id_{*}\colon X_{\Sigma'}\to X_{\Sigma}.$$

The following statements hold
	\begin{prop}\cite[Corollary 1.18]{Oda}
	Let $(\Sigma',N)$ be a smooth subdivision of $(\Sigma,N)$. Then $id_{*}\colon X_{\Sigma'}\to X_{\Sigma}$ is a proper birational mapping and is an equivariant resolution of singularities for $X_{\Sigma}$.
	\end{prop}
	\begin{thm}\cite[p. 23]{Oda}\label{resol}
Any normal toric variety  $X_{\Sigma}$ admits an equivariant resolution of singularities.
	\end{thm}
	
\subsection{Holomorphic functions in toric varieties}
	
Recall the following
\begin{defn}
A discrete valuation on a field $K$ is a group homomorphism 
 $v\colon K^{*}\to \mathbb{Z}$ that is onto and satisfies  $v(fg)=v(f)+v(g)$, $v(f+g)\geq \min(v(f),v(g))$. 
\end{defn}

Let $X_{\Sigma}$ be a smooth toric variety and $\rho\in\Sigma(1)$. Then according to Proposition \ref{orbconecoresp}, $V(\rho)=\overline{O(\rho)}$ is a $T_{N}$-invariant prime divisor that defines a discrete valuation
$$v_{\rho}\colon \mathbb{C}(T_{N})^{*}\to \mathbb{Z}$$
on the field of rational functions $\mathbb{C}(T_{N})$ on $T_{N}$ such that $v_{\rho}(f)$  is equal to the order of vanishing of the  rational function $f$ along $V(\rho)$. More information on the valuation theory in the context of toric varieties can be found in \cite[\S 4.0]{Cox}.

Let $D\subset X_{\Sigma}$ be a domain and $D\cap V(\rho)\not=\varnothing$. The valuation $v_{\rho}$ can be extended to a valuation on the field of meromorphic functions $v_{\rho}\colon \mathcal{M}(D)^{*}\to \mathbb{Z}$: if $f\in \mathcal{M}(D)^{*}$ and $x\in V(\rho)_{reg}\cap D$, then by the Weierstrass division theorem we have $f=h^{n}\phi$ locally at the point $x$, where $n\in\mathbb{Z}$ and $h_{x}$ is irreducible in $\mathcal{O}_{X_{\Sigma},x}$ such that $V(\rho)=\{h=0\}$ locally, $\phi\in\mathcal{O}_{X_{\Sigma},x} $ and $\phi|_{V(\rho)}\neq0$. We set $v_{\rho}(f):=n$,  this is the order of vanishing of $f$ on $V(\rho)$.

Let us characterize holomorphic functions in $D$.

\begin{lem}\label{lemm1} Let $D\subset X_{\Sigma}$ be domain. Then $$\mathcal{O}(D)=\{f\in\mathcal{O}(D\cap T_{N})\cap\mathcal{M}(D)\mid v_{\rho}(f)\geq 0 \forall\rho\in \Sigma(1)\; such\; that\; D\cap V(\rho)\not=\varnothing\}$$
\end{lem}

\begin{proof}
One inclusion is obvious, the converse one follows from the Riemann extension theorem.
\end{proof}

A valuation $v_\rho$ for a $\rho\in\Sigma(1)$ may be identified with a point $u_\rho$ in $N$
\begin{lem}\cite[prop. 4.1.1]{Cox}
$v_{\rho}(t^{I})=\langle u_{\rho},I\rangle$, where $u_\rho$ is the minimal integer generator of $\rho\cap N$.
\end{lem}

For any Laurent polynomial $f=\sum\limits_{I\in A} a_{I}t^{I}$ we have $v_{\rho}(f)=\min\limits_{I\in A}\langle u_{\rho},I\rangle$, which is attained for $I=I_0\in A$. Indeed,  
$$f=t^{I_{0}}\left(a_{I_0}+\sum\limits_{I\in A, I\not=I_0} a_{I}t^{I-I_0}\right)$$
and
$$v_{\rho}(f)=\langle u_{\rho},I_{0}\rangle+v_{\rho}\left(a_{I_0}+\sum\limits_{I\in A, I\not=I_0} a_{I}t^{I-I_0}\right).$$
Since  $v_{\rho}(t^{I-I_0})\geq 0$ and $a_{I_0}\neq0$, $a_{I_0}+\sum\limits_{I\in A, I\not=I_0} a_{I}t^{I-I_0}$ is a holomorphic function  at every point of $V(\rho)$ and non vanishing on $V(\rho)$. Hence, $v_{\rho}(f)=\langle u_{\rho},I_{0}\rangle$.

For series we have
\begin{lem} \label{lemm3}
If $f\in\mathcal{O}(D\cap T_{N})\cap\mathcal{M}(D)$ such that $v_{\rho}(f)\geq 0$ and is given by a Laurent series $\sum\limits_{I\in A} a_{I}t^{I}$ in a neighbourhood of $V(\rho)$, then $\langle u_{\rho},I\rangle\geq0 \,\forall I\in A$.
\end{lem}

\begin{proof}

The space $M_{\mathbb{R}}$ can be exhausted by compact sets. This induces an exhaustion of $A$ by finite sets, i.e. $A=\bigcup A_{n}$, $A_{n}\subset A_{n+1}$. The series $f=\sum\limits_{I\in A} a_{I}t^{I}$ is the limit  of partial sums $f_{n}=\sum\limits_{I\in A_{n}} a_{I}t^{I}$ in the topology of uniform convergence on compact sets in neighbourhood of $V(\rho)$. 

Assume that there exists $I\in A$ such that $\langle u_{\rho},I\rangle<0$. The monomial $t^{I}$ is a term in a partial sum $f_{n_0}$ for some $n_{0}\in \mathbb{N}$, then $v_{\rho}(f_{n_0})=\min\limits_{I\in A_{n_0}}\langle u_{\rho},I\rangle<0$. Moreover, for all $N>n_0$ we have $v_{\rho}(f_{N})<0$. This means that for all $N>n_0$ partial sums $f_{N}$ have a pole on $V(\rho)$. This contradiction proves the statement.
\end{proof}

Let $D\subset X_{\Sigma}$ be a domain such that $D\cap V(\rho)\not=\varnothing$ for all $\rho\in\Sigma(1)$.

\begin{cor}\label{mainlemm2}

If $f\in\mathcal{O}(D)$ and is given by a Laurent series $\sum\limits_{I\in A} a_{I}t^{I}$ in $D$, then $A\subset |\Sigma|^{\vee}\cap M$.
\end{cor}

\begin{proof}
By Lemma \ref{lemm1} we have $f\in\mathcal{O}(D\cap T_{N})\cap\mathcal{M}(D)$ such that $v_{\rho}(f)\geq 0$ for all $\rho\in\Sigma(1)$. The proof follows from the fact that $|\Sigma|^{\vee}=\{I\in M_{\mathbb{R}}\mid \langle u_{\rho},I\rangle\geq0 \forall \rho\in\Sigma(1) \}$ and from Lemma \ref{lemm3}.
\end{proof}

\subsection{Exhaustion of toric varieties by compact sets}\label{exhsubsect}

	\begin{defn}\label{exchcompact}
Let	$X$ be a topological space. A collection of compact subsets $\{V_{n}\}_{n=1}^{\infty}$ in $X$ is called an exhaustion of  $X$ if 
	\begin{enumerate}
	\item for any $n$ we have $V_{n}\subset \inter(V_{n+1})$;
	\item $X=\bigcup\limits_{n=1}^{\infty}V_{n}$.
	\end{enumerate}
	\end{defn}

We shall prove the existence of a certain exhaustion for toric varieties satisfying an additional property.

\begin{lem}
Let $X_{\Sigma}$ be a non-compact normal toric variety, and the complement  $\mathbb{R}^{p}\setminus |\Sigma|$ of its fan be connected. Then there exists an exhaustion of $X_{\Sigma}$ by compact sets  $\{V_n\}$ such that $X_{\Sigma}\setminus V_{n}$ are connected.
\end{lem}

\begin{proof}
Let $X_{\Sigma'}$ be a compactification of $X_{\Sigma}$, then $Z:=X_{\Sigma'}\setminus X_{\Sigma}$ is a compact algebraic set. According to \cite[Proposition 1.6]{Oda} we have $$Z=\bigcup\limits_{\tau\in\Sigma',\relint(\tau)\subset\mathbb{R}^{p}\setminus |\Sigma|}O(\tau),$$ where $\relint(\tau)$ is a relative interior of the cone $\tau$. Due to connectedness of the complement of the fan, the set $Z$ is connected.

Let $U\subset X_{\Sigma'}$ be a connected open neighborhood of $Z$. Since $U$ is also a normal space and $Z$ is a thin set in $U$, by the criterion of connectedness \cite[p. 81]{Grauert} we obtain $U\setminus Z\subset X_{\Sigma}$ is a connected set. Then the set $V=X_{\Sigma'}\setminus U$ is a compact subset in $X_{\Sigma}$, and $X_{\Sigma}\setminus V$ is connected.

A sequence of nested connected neighborhoods $\{U_{n}\}_{n=1}^{\infty}$ of $Z$ with properties $\overline{U_{n+1}}\subset U_{n}$ and $\bigcap\limits_{n}U_{n}=Z$ induces a sequence of compact sets  $\{V_{n}\}_{n=1}^{\infty}$ in $X_{\Sigma}$ giving an exhaustion of $X_{\Sigma}$ such that $X_{\Sigma}\setminus V_{n}$ is connected.

Existence of such a sequence $\{U_{n}\}_{n=1}^{\infty}$ follows from metrizable of $X_{\Sigma'}$. Indeed, let $\rho$ be a metric compatible with topology of $X_{\Sigma'}$. Then we define sets $$U_{n}:=\{x\in X_{\Sigma'}\mid \inf\limits_{z\in Z}\rho(z,x)<\frac{1}{n}\}.$$

The sequence $\{U_{n}\}_{n=1}^{\infty}$ satisfies the properties $\overline{U_{n+1}}\subset U_{n}$ and $\bigcap\limits_{n}U_{n}=Z$.

\end{proof}

\section{Cohomology with compact supports}

First we recall some facts in sheaf theory \cite{Bredon, Demailly}, and then describe the group  $H^{1}_{c}(X_{\Sigma'},\mathcal{O})$.
	
\subsection{Sheaf cohomology}
	Let $\mathcal{A}$ be a sheaf of Abelian groups on a paracompact topological space $X$. Denote by $\mathcal{A}|_F$ the restriction of  $\mathcal{A}$ to a subset $F$ of $X$.
	
	\begin{thm}\cite[Theorem 9.5]{Bredon}\label{closedandsection}
	Let $F$ be a closed subset of $X$. Then 
	 $$H^{0}(F,\mathcal{A}|_F)=\mathop{\underrightarrow{\lim}}_{U\supset F}H^{0}(U,\mathcal{A}),$$
where $U$ ranges over neighborhoods of $F$.
	\end{thm}
\noindent In other words, a section $s\in H^{0}(F,\mathcal{A}|_F)$ can be thought of as a section $s'\in H^{0}(U,\mathcal{A})$, where $U$ is some neighborhood of $F$.

\begin{defn}
	A family $\Phi$ of closed subsets of $X$ is called a family of supports if any closed subset of an element of $\Phi$ belongs to $\Phi$ and if a union of two elements of $\Phi$ belongs to $\Phi$.
\end{defn}	

For a family of supports $\Phi$ one defines cohomology with supports in $\Phi$ (see, e.g.,  \cite[Section 2, \S2]{Bredon}). If $\Phi$ is the family of compact subset of $X$ such cohomology is called cohomology with compact supports and denoted  $H^{*}_{c}(X,\mathcal{A})$. For a subset $Y$ of $X$ one can consider the groups $H^{*}_{c}(Y,\mathcal{A}|_Y)$, in this case we shall write simply $H^{*}_{c}(Y,\mathcal{A})$. If $\Phi$ is the family of closed subsets of a compact subset $K$ of $X$, we get cohomology with supports in $K$ and denote it by $H^{*}_{K}(X,\mathcal{A})$ 

If $\{V_n\}_{n=1}^{\infty}$ is a compact exhaustion of $X$ then there is the following relation:
\begin{prop}\cite[p. 11]{BanStan}\label{exhlemm}
There is canonical isomorphism
$$\mathop{\underrightarrow{\lim}}_{V_{n}}H^{p}_{V_n}(X,\mathcal{F})\cong H^{p}_{c}(X,\mathcal{F}).$$
\end{prop}

For later use we need two long exact sequences.
	\begin{thm}\cite[Section 2, \S 10]{Bredon}\label{pairexact}
	Let $X$ be locally compact, $F\subset X$ be closed, and $U=X\setminus F$. Then the following sequence is exact
	\begin{equation*}
	\xymatrix@C=0.5cm{
\cdots \ar[r] &  H^{p}_{c}(U,\mathcal{A}) \ar[r] &  H^{p}_{c}(X,\mathcal{A}) \ar[r] & H^{p}_{c}(F,\mathcal{A}) \ar[r] & H^{p+1}_{c}(U,\mathcal{A}) \ar[r] & \cdots }
	\end{equation*}
	\end{thm}
	\begin{thm}\cite[pp. 52-53]{BanStan}\label{propmain1}
Let $K$ be a closed subset of $X$, then the following sequence is exact

\begin{equation*}
	\xymatrix@C=0.5cm{
0 \ar[r] &  H^{0}_{K}(X,\mathcal{F}) \ar[r] &  H^{0}(X,\mathcal{F}) \ar[r] & H^{0}(X\setminus K,\mathcal{F}) \ar[r] & H^{1}_{K}(X,\mathcal{F}) \ar[r] & \cdots }
\end{equation*}
\end{thm}

\subsection{The group $H^{1}_{c}(X_{\Sigma'},\mathcal{O})$}
	Before giving a description of this cohomology group, we prove a lemma.
	\begin{lem}\label{mainlemm1}
Let $Z\subset\mathbb{C}^{n}$ be a union of coordinate subspaces, and  $U$ be its neighborhood. Then there exists a complete Reinhardt domain $W$ centered at the origin such that $Z\subset W\subset U$.
  \end{lem}
  \begin{proof}
  For a subset $J\subset [n]=\{1,\cdots,n\}$ we denote by $\mathbb{C}^{|J|}_{z_J}$ the $|J|$-dimensional coordinate subspace with coordinates $z_J=(z_{i}\mid i\in J)$. Then $Z=\bigcup\limits_{J\in F}\mathbb{C}^{|J|}_{z_J}$ where $F$ is some set of subsets of $[n]$. 
    
  Consider a point $z=(z_1,\cdots,z_{n})\in Z$, then $z\in\mathbb{C}^{|J|}_{z_{J}} \subset U$ for some $J\in F$. There exists an open polycylinder  $$W_{z}:=\{w=(w_1,\cdots, w_n)\in\mathbb{C}^{n}\mid |w_{i}|< e_{i}\forall i\in [n] - J \text{ and } \{|w_{i}|<|z_{i}|\forall i\in  J\}$$ such that for sufficiently small $e_{i}$ this polycylinder is contained in $U$. 
  
Indeed, let $A:=\{|w_{i}|<|z_{i}|\mid i\in J\}\subset \mathbb{C}^{|J|}_{z_{J}}$ and consider  $\rho(A,\partial U)=\inf\limits_{a\in A, b\in\partial U}\|a-b\|$ (here $\|a-b\|:=\sqrt{\sum\limits_{i} |a_{i}-b_{i}|^{2}}$ is the Euclidean metric in $\mathbb{C}^{n}$). Then we let $e_{i}\leq\frac{\rho(A,\partial U)}{2\sqrt{n-|J|}},\forall i\in [n]-J$. 

Assume that $w\in W_{z}$ but $w\not\in U$, then $\rho(A,\partial U)\leq \rho(A,w)$ (indeed, $\rho(A,\partial U)\leq\rho(A,b)$ for all points $b\in \partial U$, in particular we have $\rho(A,\partial U)\leq\rho(A,\partial U\cap L)\leq \rho(A,w)$ where $L$ is the perpendicular to the $A$ from the point $w$). We have  $w=(w_{[n]-J},w_{J})\in \mathbb{C}^{n-|J|}_{z_{[n]-J}} \times\mathbb{C}^{|J|}_{z_{J}} $ then $w=w_1+w_2$ where $w_{1}=(w_{[n]-J},0), w_{2}=(0, w_{J})$. Since $\|w-a\|\leq\|w_1\|+\|w_{2}-a\|$ for $a\in A$, we obtain 
\begin{multline*}
\inf\limits_{a\in A}\|w-a\|\leq \inf\limits_{a\in A}(\|w_1\|+\|w_{2}-a\|)=\|w_1\|=\\=\sqrt{\sum\limits_{i\in [n]-J} |w_{i}|^{2}}\leq \sqrt{\sum\limits_{i\in [n]-J} e_{i}^{2}}\leq\frac{\rho(A,\partial U)}{2}.
\end{multline*}

Therefore $\rho(A,\partial U)\leq\frac{\rho(A,\partial U)}{2}$, so we have arrived at a contradiction, and  $W_{z}\subset U$. 
  
Having done the same for all $z\in Z$, we get a family of polycylinders $\{W_{z}\}_{z\in Z}$ centered at the origin. Their union $W:=\bigcup\limits_{z\in Z}W_{z}$ is an open complete Reinhardt domain centered at the origin that contains the set $Z$.
  \end{proof}
This implies that if $f\in\mathcal{O}(U)$ then a smaller neighbourhood $U'\subset W\subset U$ the function is represented by an absolutely convergent Taylor series.

	Now we turn to the description of $H^{1}_{c}(X_{\Sigma'},\mathcal{O})$. Here we use a long exact sequence of a pair, as Marciniak did in \cite{Marci2} for smooth toric surfaces.
	
	Let  $N=\mathbb{Z}^{p}$, $\Sigma'$ be a fan with one connected component of its complement, and $X_{\Sigma'}$ the corresponding toric variety. According to Corollary \ref{corequvcompact} from Theorem \ref{equvcompact} the fan $\Sigma'$ can be completed to become a complete fan $\Sigma''$. Denote by $\Sigma$ the fan that consists of those cones of $\Sigma''$ whose supports lie in the closed set $|\Sigma''|\setminus\inter(|\Sigma'|)$. 
	
The variety $X_{\Sigma'}$ is an open subspace the compact complex space $X_{\Sigma''}$. The complement $Z:=X_{\Sigma''}\setminus X_{\Sigma'}$ is a connected compact $T_N$-invariant analytic set in $X_{\Sigma''}$. 
	
	By Theorem \ref{pairexact} we have 
	\begin{multline}\label{seqshort}
\xymatrix@C=0.5cm{
  0 \ar[r] & H^{0}_{c}(X_{\Sigma'},\mathcal{O}) \ar[rr] && H_{c}^{0}(X_{\Sigma''},\mathcal{O}) \ar[rr] && H^{0}_{c}(Z,\mathcal{O}) \ar[r] &\\ \ar[r] & H^{1}_{c}(X_{\Sigma'},\mathcal{O}) \ar[rr] && H_{c}^{1}(X_{\Sigma''},\mathcal{O}) \ar[rr] && H^{1}_{c}(Z,\mathcal{O}) \ar[r] & \cdots }
	\end{multline}
	Since $X_{\Sigma'}$ is not compact, $H^{0}_{c}(X_{\Sigma'},\mathcal{O})=0$. Moreover, on a compact space $X_{\Sigma''}$ we have $H_{c}^{*}(X_{\Sigma''},\mathcal{O})=H_{}^{*}(X_{\Sigma''},\mathcal{O})$, which is $\mathbb C$ in dimension 0 and vanishes for higher dimensions \cite[Corollary 2.9]{Oda}. Therefore,  $H_{c}^{0}(X_{\Sigma''},\mathcal{O})=\mathbb{C}$ and $H_{c}^{1}(X_{\Sigma''},\mathcal{O})=0$. From (\ref{seqshort}) we get 
	\begin{equation*}
\xymatrix@C=0.5cm{
  0 \ar[r] & \mathbb{C} \ar[rr] && H^{0}(Z,\mathcal{O}) \ar[rr] && H^{1}_{c}(X_{\Sigma'},\mathcal{O}) \ar[r] & 0 }.
    \end{equation*}
Thus, $H^{1}_{c}(X_{\Sigma'},\mathcal{O})\cong H^{0}(Z,\mathcal{O})/\mathbb{C}$. On the other hand, by Theorem \ref{closedandsection} $$H^{0}(Z,\mathcal{O})=\mathop{\underrightarrow{\lim}}\limits_{U\supset Z}H^{0}(U,\mathcal{O}).$$ Since the set $Z$ together with its neighborhood lies in  $X_{\Sigma}$, we can work with $X_{\Sigma}$ instead of $X_{\Sigma''}$. 

By Proposition \ref{orbconecoresp} $$Z=\bigcup\limits_{\tau\in\Sigma,\relint(\tau)\in\inter(|\Sigma|)}O(\tau)\subset X_{\Sigma},$$ 
where $O(\tau)$ is the orbit of the action of $(\mathbb{C}^{*})^{n}$ on $X_{\Sigma}$ that corresponds to a cone $\tau$.

Assume that  $X_{\Sigma}$ is smooth and consider the equivalence class $[f,V]\in H^{0}(Z,\mathcal{O})$, that is a function $f$ holomorphic in a neighborhood  $V$ of $Z$. 
  
  Recall that $$Z\cap U_{\sigma}=\bigcup\limits_{\tau<\sigma,\relint(\tau)\in\inter(|\Sigma|)}O(\tau)\subset X_{\Sigma},$$
where $U_{\sigma}$ is an affine chart corresponding to $\sigma\in \Sigma$, i.e. $Z\cap U_{\sigma}$ is a union of coordinate subspaces in $U_{\sigma}\cong\mathbb{C}^{p}$.
  
By Lemma \ref{mainlemm1}, for $V\cap U_{\sigma}$ there exists a complete Reinhardt domain $W_{\sigma}$ such that $Z\cap U_{\sigma}\subset W_{\sigma} \subset V\cap U_{\sigma}$. Therefore in $W_{\sigma}$ the function $f$ is given by a convergent power series. Choose a neighborhood $D$ of $Z$ such that $D\subset\bigcup\limits_{\sigma}W_{\sigma}$ and $D\cap U_{\sigma}\subset W_{\sigma}$. In $D\cap U_{\sigma}$ the function $f$ is given by the same series as in $W_{\sigma}$. All these series written in the coordinates $t$ of the torus $T_{N}$ are one and the same power series; thus, $f=\sum\limits_{I\in A}a_{I}t^{I}$. By Corollary \ref{mainlemm2} we get $A\subset|\Sigma|^{\vee}\cap M$. Thus, an equivalence class $[f,V]$ can be written as $[\sum\limits_{I\in A}a_{I}t^{I},D]$, provided the series converges in $D$. Denote now $C:=\mathbb{R}^{p}\setminus |\Sigma'|$, then $\overline{C}=|\Sigma|$. In this notation we have that $$H^{0}(Z,\mathcal{O})=\{[\sum\limits_{I\in A}a_{I}t^{I},D]\mid\text{ the series converges in } D, A\subset \overline{C}^{\vee}\cap M \}.$$

Assume now that $X_{\Sigma}$ is an arbitrary (normal) toric variety, not necessarily smooth. We reduce this case to the smooth one. Let $\Sigma_{0}$ be a smooth subdivision of  $\Sigma$ and $\pi\colon X_{\Sigma_{0}}\to X_{\Sigma}$ be an equivariant resolution of singularities. Note that $$\pi^{-1}(Z)=\bigcup\limits_{\tau\in\Sigma_0,\relint(\tau)\in\inter(|\Sigma_0|)}O(\tau)\subset X_{\Sigma_0}.$$ Denote this set by $Z_{0}$.

For structure sheaves  $\mathcal{O}_{\Sigma}, \mathcal{O}_{\Sigma_{0}}$ of $X_{\Sigma}, X_{\Sigma_{0}}$, respectively, we have a natural mapping $$H^{0}(Z, \mathcal{O}_{\Sigma})\to H^{0}(Z_0, \mathcal{O}_{\Sigma_{0}}), [f, V]\to [f\circ\pi, \pi^{-1}(V)].$$
This is an isomorphism. Injectivity is obviuous, let us show its surjectivity. Consider a class $[g,U]\in H^{0}(Z_0, \mathcal{O}_{\Sigma_{0}})$, let $T_{N}$ be the torus in $X_{\Sigma}$ and $T_{N}'$ be the torus in $X_{\Sigma_0}$. There exists a neighborhood $V\subset X_{\Sigma}$ of $Z$ such that $\pi'=\pi|_{U\cap T_{N}'}\colon U\cap T_{N}'\cong V\cap T_{N}$. There exists $f\in H^{0}(T_{N}\cap V,\mathcal{O})$ such that $f\circ\pi'=g|_{T_{N}'\cap U}$. Since $f\circ\pi'$ is locally bounded in $U\cap V(\tau),\forall \tau\in\Sigma_{0}(1)$ and $\pi(\bigcup\limits_{\tau\in\Sigma_{0}(1)}V(\tau))=\bigcup\limits_{\tau\in\Sigma(1)}V(\tau)$, then $f$ is locally bounded in $V\cap V(\tau),\forall \tau\in\Sigma(1)$. By the Riemann extension theorem the function $f$ can be holomorphically extended to $V$. So, we have holomorphic functions $f\circ\pi \in H^{0}(\pi^{-1}(V),\mathcal{O}_{\Sigma_0})$ and $g\in H^{0}(U,\mathcal{O}_{\Sigma_0})$ such that $f\circ\pi\equiv g$ on $\pi^{-1}(V)\cap T_{N}'=U\cap T_{N}'$. By the uniqueness theorem we obtain $f\circ\pi\equiv g$ on $\pi^{-1}(V)\cap U$, therefore $[g,U]=[g|_{\pi^{-1}(V)\cap U},\pi^{-1}(V)\cap U]=[f\circ\pi, \pi^{-1}(V)\cap U]=[f\circ\pi, \pi^{-1}(V)]$. 

This proves the following 
  	\begin{lem}\label{mainlemm} Let $X_{\Sigma'}$ be a normal toric variety with the set $C:=\mathbb{R}^p\setminus |\Sigma'|$ being connected. Then
$$H^{1}_{c}(X_{\Sigma'},\mathcal{O})\cong\{[\sum\limits_{I\in A}a_{I}t^{I},D]\mid A\subset \overline{C}^{\vee}\cap M \}/\mathbb{C}.$$
where the series converges in some neighborhood $D\supset Z$ of the smooth variety $X_{\Sigma}$. 
	\end{lem}
	
	Recall that a set $A\subset \mathbb{R}^{n}$ with the property that for any point $x\in A$ and $\lambda\geq0$ one has $\lambda x\in A$ is called non-negatively homogeneous or simply a cone (not necessarily convex), see \cite[p. 25]{Boyd}). 
    \begin{lem}\cite[p. 53]{Boyd}\label{convgeo1}
Let $A$ be a non-negatively homogeneous set, then $$\overline{\conv(A)}=A^{\vee\vee}$$
In particular, if $A$ is closed then $\conv(A)=A^{\vee\vee}$.
    \end{lem}
    
This gives us the following result
	
	\begin{thmA*}
	Let $X_{\Sigma'}$ be a $p$-dimensional toric variety with the fan $\Sigma'$. Assume that the complement of the fan's support  $C:=\mathbb{R}^{p}\setminus |\Sigma'|$ is connected, then  $H^{1}_{c}(X_{\Sigma'},\mathcal{O})=0$ if and only if $\conv(\overline{C})=\mathbb{R}^{p}$.
	\end{thmA*}	
	\begin{proof}
Since $\overline{C}$ is a non-negatively homogeneous set, by Lemma \ref{convgeo1} $\conv(\overline{C})=\overline{C}^{\vee\vee}$. Therefore $\conv(\overline{C})=\mathbb{R}^{p}$ if and only if $\overline{C}^{\vee}=O$. By Lemma \ref{mainlemm}, $\overline{C}^{\vee}=O$ if and only if $H^{1}_{c}(X_{\Sigma'},\mathcal{O})=0$.
	\end{proof}
	
\section{The Hartogs phenomenon}	
	
	\subsection{Toric Serre's theorem}
	
J.-P. Serre has formulated the cohomological condition for the Hartogs phenomenon in complex manifolds \cite{Serre}. We give a proof of this statement for toric varieties. First, recall the excision lemma \cite[pp. 52-53]{BanStan}
	
\begin{lem}\label{propmain2}
Let $X$ be a topological space. If $Y$ is an open subset in $X$ and $A$ is a closed subset of $Y$, then we have canonical isomorphisms $H^{n}_{A}(X,\mathcal{F})\cong H^{n}_{A}(Y,\mathcal{F})$ for any $n$.
\end{lem}

Now, we can prove toric Serre's theorem.
\begin{thmB*}
 	Let $X_{\Sigma}$ be a non-compact normal toric variety with the complement $\mathbb{R}^{p}\setminus |\Sigma|$ being connected. The cohomology group $H^{1}_{c}(X_{\Sigma},\mathcal{O})$ is trivial if and only if  $X_{\Sigma}$ admits the Hartogs phenomenon.
\end{thmB*}

\begin{proof}
By \eqref{exhsubsect} there exists an exhaustion of $X_{\Sigma}$ by compact sets $\{V_{n}\}$ such that $X_{\Sigma}\setminus V_{n}$ is connected. By Lemma \ref{exhlemm} it follows that $H^{1}_{c}(X_{\Sigma},\mathcal{O})\cong \mathop{\underrightarrow{\lim}}\limits_{V_n}H^{1}_{V_n}(X_{\Sigma},\mathcal{O})$.

Suppose that $H^{1}_{c}(X_{\Sigma},\mathcal{O})=0$. Let $X_{\Sigma'}$ be a toric compactification of $X_{\Sigma}$. By vanishing theorem for compact toric varieties \cite[Corollary 2.9]{Oda} and Theorem \ref{propmain1}, we obtain a short exact sequence for any $n$: 

\begin{multline*}
\xymatrix@C=0.5cm{
  0 \ar[r] & \mathbb{C} \ar[rr] && H^{0}(X_{\Sigma'}\setminus V_{n},\mathcal{O}) \ar[rr] && H^{1}_{V_n}(X_{\Sigma},\mathcal{O}) \ar[r] & 0 }.
\end{multline*}

Using homological algebra machinery (see, e.g. \cite[Section 4, \S1]{Demailly}), we get the following commutative diagram  

	\begin{equation*}
\begin{CD}\phantom{gggg} @VVV @VVV @VVV\\
0 \to\mathbb{C} @>>> H^{0}(X_{\Sigma'}\setminus V_{n-1},\mathcal{O}) @>>>  H^{1}_{V_{n-1}}(X_{\Sigma},\mathcal{O})\to 0\\
\phantom{gggg} @V{id_{n-1}}VV @V{p_{n-1}}VV @V{q_{n-1}}VV\\
0 \to\mathbb{C} @>>> H^{0}(X_{\Sigma'}\setminus V_{n},\mathcal{O}) @>>>  H^{1}_{V_n}(X_{\Sigma},\mathcal{O})\to 0\\
\phantom{gggg} @V{id_{n}}VV @V{p_{n}}VV @V{q_n}VV \\
0 \to \mathbb{C} @>>> H^{0}(X_{\Sigma'}\setminus V_{n+1},\mathcal{O}) @>>>  H^{1}_{V_{n+1}}(X_{\Sigma},\mathcal{O}) \to 0 \\
\phantom{gggg} @VVV @VVV @VVV
\end{CD}.
	\end{equation*}
	
Here $id_n$ is the identity homomorphism, $p_{n}$ is the restriction homomorphism and $q_{n}$ is the homomorphism induced by the embedding $V_{n}\subset V_{n+1}$. 

Since $\mathop{\underrightarrow{\lim}}\limits_{V_n}H^{1}_{V_n}(X_{\Sigma},\mathcal{O})=0$, taking direct limits of the diagram, we obtain $\mathop{\underrightarrow{\lim}}\limits_{V_n}H^{0}(X_{\Sigma'}\setminus V_{n},\mathcal{O})\cong\mathbb{C}$. Also we have $H^{0}(X_{\Sigma'}\setminus V_{n},\mathcal{O})= H^{0}(U_{n},\mathcal{O})$, where $Z:=X_{\Sigma'}\setminus X_{\Sigma}$ and  $U_{n}:=X_{\Sigma'}\setminus V_{n}$ is a connected neighborhood of $Z$. 

Now, let $f\in H^{0}(U_{n},\mathcal{O})$. Since $\mathop{\underrightarrow{\lim}}\limits_{U_n}H^{0}(U_{n},\mathcal{O})\cong\mathbb{C}$, we have $f\equiv const$ in $U_n$ by the uniqueness theorem. It follows that $H^{0}(U_{n},\mathcal{O})=\mathbb{C}$ and for any compact set $V_{n}$ we obtain $H^{1}_{V_n}(X_{\Sigma},\mathcal{O})=0$. 

Let $K$ be a compact set in the domain $D$ such that $D\setminus K$ is connected. There exists a compact set $V_n$ such that $K\subset V_n$. It follows that $$H^{0}(X_{\Sigma'}\setminus K,\mathcal{O})\subset H^{0}(X_{\Sigma'}\setminus V_{n},\mathcal{O}).$$ Therefore $H^{1}_{K}(X_{\Sigma},\mathcal{O})=0$, and by Lemma \ref{propmain2} we obtain $H^{1}_{K}(D,\mathcal{O})=0$. By Theorem \ref{propmain1} the restriction map $H^{0}(D,\mathcal{O})\to H^{0}(D\setminus K,\mathcal{O})$ is an isomorphism, i.e. the Hartogs phenomenon holds in  $X_{\Sigma}$.

Now, suppose that the Hartogs phenomenon holds in  $X_{\Sigma}$, i.e. for any domain $D\subset X_{\Sigma}$ and for any compact set $K\subset D$ such that $D\setminus K$ is connected, the restriction map $H^{0}(D,\mathcal{O})\to H^{0}(D\setminus K,\mathcal{O})$  is an isomorphism.  

We have a commutative diagram where all morphisms are restriction maps: 
	
	\begin{equation*}
\begin{CD}
H^{0}(X_{\Sigma'}, \mathcal{O}) @>>> H^{0}(X_{\Sigma'}\setminus K,\mathcal{O}) \\
@VVV @VVV\\
H^{0}(D, \mathcal{O}) @>>> H^{0}(D\setminus K,\mathcal{O})
\end{CD}
	\end{equation*}
	
By assumption, the lower arrow of the diagram is an isomorphism. Let $f\in H^{0}(X_{\Sigma'}\setminus K,\mathcal{O})$ be a holomorphic function, and let $h:=f|_{D\setminus K}$. There exists $g\in H^{0}(D, \mathcal{O})$ such that $g\equiv h$ on $D\setminus K$. Since $X_{\Sigma'}=(X_{\Sigma'}\setminus K)\cup D$ and $D\setminus K=(X_{\Sigma'}\setminus K)\cap D$, there exists $F\in H^{0}(X_{\Sigma'}, \mathcal{O})$ such that $F\equiv f$ on $X_{\Sigma'}\setminus K$ and $F\equiv g$ on $D$. But $H^{0}(X_{\Sigma'}, \mathcal{O})=\mathbb{C}$, therefore $H^{0}(X_{\Sigma'}\setminus K,\mathcal{O})=\mathbb{C}$ and $H_{K}^{1}(X_{\Sigma},\mathcal{O})=0$. 
	
Now we take an exhaustion of $X_{\Sigma}$ by compact sets $\{V_{n}\}$ such that each $X_{\Sigma}\setminus V_{n}$ is connected. Replacing $K$ by $V_n$ in above, we obtain $H_{V_{n}}^{1}(X_{\Sigma},\mathcal{O})=0$. By Lemma \ref{exhlemm} it follows that $H^{1}_{c}(X_{\Sigma},\mathcal{O})\cong \mathop{\underrightarrow{\lim}}\limits_{V_n}H^{1}_{V_n}(X_{\Sigma},\mathcal{O})=0$.
\end{proof}

	\subsection{The Hartogs phenomenon for normal toric varieties}

We can now prove our main result from which Marciniak's conjecture follows.

	\begin{thmC*}
Let $X_{\Sigma}$ be a normal non-compact toric variety with the fan $\Sigma$ whose complement is $\mathbb{R}^{p}\setminus |\Sigma|=\bigsqcup\limits_{j=1}^{n}C_{j}$. Then
	\begin{itemize}
	\item if at least one of $C_j$'s is concave then $X_{\Sigma}$ admits the Hartogs phenomenon.
	\item if $n=1$ then the converse is also true, i.e. if $X_{\Sigma}$ admits the Hartogs phenomenon then $\mathbb{R}^p \setminus |\Sigma|$ is concave.
	\end{itemize}  
	\end{thmC*}
	
	\begin{proof}

Let $C_{1}$ be a concave component. The toric variety $X_{\Sigma}$ can be embedded into a toric variety $X_{\Sigma'}$ such that $\mathbb{R}^{p}\setminus |\Sigma'|=C_{1}$. By Theorem A we have $H^{1}_{c}(X_{\Sigma'},\mathcal{O})=0$. Therefore, by toric Serre's theorem (Theorem B), the Hartogs phenomenon holds in $X_{\Sigma'}$. Since $X_{\Sigma}\subset X_{\Sigma'}$ is an open subvariety, this is true in $X_{\Sigma}$ also. 

If $n=1$ then by toric Serre's theorem, $H^{1}_{c}(X_{\Sigma},\mathcal{O})=0$. It remains to apply Theorem A to finish the proof.
	\end{proof}


\end{document}